\documentclass{amsart}

\usepackage{amsmath,
            amssymb,
            bibentry,
            enumitem,
            graphicx,
            mathrsfs,
            tabu,
            xypic}
\usepackage[comma,numbers,sort,square]{natbib}
\usepackage[colorlinks=true,
            linktocpage=true,
            linkcolor=magenta,
            citecolor=magenta,
            urlcolor=magenta]{hyperref}
\PassOptionsToPackage{hyphens}{url}

\newtheorem{theorem}{Theorem}[section]

\newtheorem{lemma}[theorem]{Lemma}
\newtheorem{corollary}[theorem]{Corollary}
\theoremstyle{definition}
\newtheorem{definition}[theorem]{Definition}

\newtheorem{observation}[theorem]{Observation}

\usepackage{xcolor}	
\usepackage{soul}

\newtheorem{innerhighlight}{}
\newenvironment{highlight}[1]
  {\renewcommand\theinnerhighlight{#1}\innerhighlight}
  {\endinnerhighlight}

\newcommand{\ran}{\operatorname{ran}}
\newcommand{\dom}{\operatorname{dom}}

\newcommand{\seq}[1]{\langle #1 \rangle}

\newcommand{\ZFC}{\mathsf{ZFC}}

\newcommand{\RCA}{\mathsf{RCA}}

\newcommand{\RT}{\mathsf{RT}}
\newcommand{\COH}{\mathsf{COH}}

\newcommand{\SRT}{\mathsf{SRT}}

\newcommand{\D}{\mathsf{D}}

\newcommand{\cred}{\leq_{\text{\upshape c}}}
\newcommand{\ncred}{\nleq_{\text{\upshape c}}}
\newcommand{\scred}{\leq_{\text{\upshape sc}}}
\newcommand{\nscred}{\nleq_{\text{\upshape sc}}}
\newcommand{\ured}{\leq_{\text{\upshape W}}}
\newcommand{\nured}{\nleq_{\text{\upshape W}}}
\newcommand{\sured}{\leq_{\text{\upshape sW}}}
\newcommand{\nsured}{\nleq_{\text{\upshape sW}}}

\newcommand{\trunc}[1]{{#1}^{\#}}

\newcommand{\LimSet}{I_d}

\usepackage{xcolor}	
\usepackage{soul}
\definecolor{grey}{rgb}{0.95,0.95,0.95}
\definecolor{green}{rgb}{0.1,0.9,0.2}

\begin{document}

\title[Ramsey's theorem and strong computable reducibility]{Ramsey's theorem for singletons and strong computable reducibility}


\author[D.\ D.\ Dzhafarov]{Damir D. Dzhafarov}
\address{Department of Mathematics\\
University of Connecticut\\
Storrs, Connecticut U.S.A.}
\email{damir@math.uconn.edu}

\author[L.\ Patey]{Ludovic Patey}
\address{Laboratoire PPS\\
Université Paris Diderot\\
Paris, France}
\email{ludovic.patey@computability.fr}

\author[R.\ Solomon]{Reed Solomon}
\address{Department of Mathematics\\
University of Connecticut\\
Storrs, Connecticut U.S.A.}
\email{david.solomon@uconn.edu}

\author[L.\ B.\ Westrick]{Linda Brown Westrick}
\address{Department of Mathematics\\
University of Connecticut\\
Storrs, Connecticut U.S.A.}
\email{linda.westrick@uconn.edu}

\thanks{Dzhafarov was partially supported by NSF grant DMS-1400267. Patey was funded by the John Templeton Foundation (`Structure and Randomness in the Theory of Computation' project). The opinions expressed in this publication are those of the authors and do not necessarily reflect the views of the John Templeton Foundation. The authors are grateful to the anonymous referee for a number of helpful comments and suggestions.}

\maketitle

\begin{abstract}
	We answer a question posed by Hirschfeldt and Jockusch by showing that whenever $k > \ell$, Ramsey's theorem for singletons and $k$-colorings, $\RT^1_k$, is not strongly computably reducible to the stable Ramsey's theorem for $\ell$-colorings, $\SRT^2_\ell$. Our proof actually establishes the following considerably stronger fact: given $k > \ell$, there is a coloring $c : \omega \to k$ such that for every stable coloring $d : [\omega]^2 \to \ell$ (computable from $c$ or not), there is an infinite homogeneous set $H$ for $d$ that computes no infinite homogeneous set for $c$. This also answers a separate question of Dzhafarov, as it follows that the cohesive principle, $\COH$, is not strongly computably reducible to the stable Ramsey's theorem for all colorings, $\SRT^2_{<\infty}$. The latter is the strongest partial result to date in the direction of giving a negative answer to the longstanding open question of whether $\COH$ is implied by the stable Ramsey's theorem in $\omega$-models of $\RCA_0$.
\end{abstract}

\section{Introduction}

In this paper, we answer several questions pertaining to the logical content of Ramsey's theorem, and thus we open by recalling the statement of this principle. For an infinite set $X$ and $n \geq 1$, let $[X]^n$ denote the set of all tuples $\seq{x_0,\ldots,x_{n-1}} \in X^n$ with $x_0 < \cdots < x_{n-1}$. For $k \geq 1$, a \emph{$k$-coloring} of $[X]^n$ is a map $c : [X]^n \to \{0,\ldots,k-1\}$, which we abbreviate as $c : [X]^n \to k$. For $\seq{x_0,\ldots,x_{n-1}} \in [X]^n$, we write $c(x_0,\ldots,x_{n-1})$ instead of $c(\seq{x_0,\ldots,x_{n-1}})$. A set $Y \subseteq X$ is \emph{homogeneous} for $c$ if there is an $i < k$ such that $c(y_0,\ldots,y_{n-1}) = i$ for all $\seq{y_0,\ldots,y_{n-1}} \in [Y]^n$. In this case, we say also that $Y$ is homogeneous \emph{with color $i$}. In other words, $Y$ is homogeneous for $c$ if $[Y]^n$ is monochromatic for $c$.

\begin{highlight}{Ramsey's theorem for $n$-tuples and $k$-colorings ($\RT^n_k$)}
	For every coloring $c : [\omega]^n \to k$, there is an infinite set $H$ which is homogeneous for $c$.
\end{highlight}

\noindent Understanding the effective and proof-theoretic content of $\RT^n_k$, for various values of $n$ and $k$, has long been a major driving force of research in computability theory, reverse mathematics, and their intersection. The traditional approach in computability theory has been to measure the complexity of homogeneous sets of computable colorings, in terms of Turing degrees and the various inductive hierarchies involving them. In reverse mathematics, the relationship of Ramsey's theorem to various other principles, from combinatorics as well as other areas, has been investigated, most commonly in the sense of which principles are implied by Ramsey's theorem, and which imply it, over the weak subsystem $\RCA_0$ of second-order arithmetic. We refer the reader to Soare~\cite{Soare-TA} and Simpson~\cite{Simpson-2009} for general background on computability and reverse mathematics, respectively, and to Hirschfeldt~\cite[Section 6]{Hirschfeldt-2014} for a comprehensive survey of results about Ramsey's theorem specifically. We refer to Shore \cite[Chapter 3]{Shore-TA} for background on forcing in arithmetic.

\subsection{Notions of computability-theoretic reduction}

As is well-known, there is a natural interplay between the two endeavors described above, with each of the benchmark subsystems of second-order arithmetic broadly corresponding to a particular level of computability-theoretic complexity (see, e.g.,~\cite[Section 1]{HS-2007} for details). But more is true. The majority of principles one considers in reverse mathematics, like Ramsey's theorem, have the syntactic form
\[
	\forall X\, (\Phi(X) \to \exists Y \, \Psi(X,Y)),
\]
where $\Phi$ and $\Psi$ are arithmetical predicates. It is common to call such a principle a \emph{problem}, and to call each $X$ such that $\Phi(X)$ holds an \emph{instance} of this problem, and each $Y$ such that $\Psi(X,Y)$ holds a \emph{solution} to $X$. The instances of $\RT^n_k$ are thus the colorings $c : [\omega]^n \to k$, and the solutions to any such $c$ are the infinite homogeneous sets for this coloring. While an implication over $\RCA_0$ between problems, say $\mathsf{Q} \to \mathsf{P}$, can in principle make multiple applications of the antecedent $\mathsf{Q}$, or split into cases in a non-uniform way, and generally be quite complicated, in practice, most implications have a considerably simpler shape. Define the following notions of reduction between problems.

\begin{definition}\label{def:reds}
	Let $\mathsf{P}$ and $\mathsf{Q}$ be problems.	
	\begin{enumerate}
		\item $\mathsf{P}$ is \emph{computably reducible} to $\mathsf{Q}$, written $\mathsf{P} \cred \mathsf{Q}$, if every instance $X$ of $\mathsf{P}$ computes an instance $\widehat{X}$ of $\mathsf{Q}$, such that if $\widehat{Y}$ is any solution to $\widehat{X}$ then there is a solution $Y$ to $X$ computable from $X \oplus \widehat{Y}$.
		\item $\mathsf{P}$ is \emph{strongly computably reducible} to $\mathsf{Q}$, written $\mathsf{P} \scred \mathsf{Q}$, if every instance $X$ of $\mathsf{P}$ computes an instance $\widehat{X}$ of $\mathsf{Q}$, such that if $\widehat{Y}$ is any solution to $\widehat{X}$ then there is a solution $Y$ to $X$ computable from $\widehat{Y}$.
		\item $\mathsf{P}$ is \emph{Weihrauch reducible} to $\mathsf{Q}$, written $\mathsf{P} \ured \mathsf{Q}$, if there are Turing functionals $\Phi$ and $\Delta$ such that if $X$ is any instance of $\mathsf{P}$ then $\Phi^X$ is an instance of $\mathsf{Q}$, and if $\widehat{Y}$ is any solution to $\mathsf{Q}$ then $\Delta^{X \oplus \widehat{Y}}$ is a solution to $X$.
		\item $\mathsf{P}$ is \emph{strongly Weihrauch reducible} to $\mathsf{Q}$, written $\mathsf{P} \sured \mathsf{Q}$, if there are Turing functionals $\Phi$ and $\Delta$ such that if $X$ is any instance of $\mathsf{P}$ then $\Phi^X$ is an instance of $\mathsf{Q}$, and if $\widehat{Y}$ is any solution to $\mathsf{Q}$ then $\Delta^{\widehat{Y}}$ is a solution to $X$.
	\end{enumerate}
\end{definition}
\noindent All of these reductions express the idea of taking a problem, $\mathsf{P}$, and computably (even uniformly computably, in the case of $\ured$ and $\sured$) transforming it into another problem, $\mathsf{Q}$, in such a way that being able to solve the latter computably (uniformly computably) tells us how to solve the former. This is a natural idea, and indeed, more often than not an implication $\mathsf{Q} \to \mathsf{P}$ over $\RCA_0$ (or at least, over $\omega$-models of $\RCA_0$) is a formalization of some such reduction. The strong versions above may appear more contrived, since it does not seem reasonable to deliberately bar access to the instance of the problem one is working with. Yet commonly, in a reduction of the above sort, the ``backward'' computation from $\widehat{Y}$ to $Y$ turns out not to reference the original instance. Frequently, it is just the identity.

Let $\mathsf{P} \leq_\omega \mathsf{Q}$ denote that every $\omega$-model of $\mathsf{Q}$ is a model of $\mathsf{P}$. It is easy to see that the following implications hold:
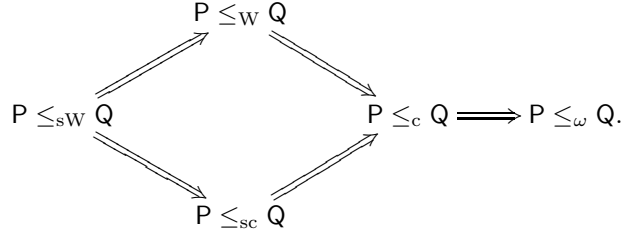
\[
\xymatrix{
& \mathsf{P} \ured \mathsf{Q} \ar@2[dr]\\
\mathsf{P} \sured \mathsf{Q} \ar@2[ur] \ar@2[dr] & & \mathsf{P} \cred \mathsf{Q} \ar@2[r] & \mathsf{P} \leq_\omega \mathsf{Q}.\\
& \mathsf{P} \scred \mathsf{Q} \ar@2[ur]
}
\]
No additional arrows can be added to this diagram (see~\cite[Section 1]{HJ-TA}). The notions of computable reducibility and strong computable reducibility were implicitly used in many papers on reverse mathematics, but were first isolated and studied for their own sake by Dzhafarov~\cite{Dzhafarov-2015}, and also form the basis of the iterated forcing constructions of Lerman, Solomon, and Towsner \cite{LST-2013}. Weihrauch reducibility (also called \emph{uniform reducibility}) and strong Weihrauch reducibility were introduced by Weihrauch~\cite{Weihrauch-1992}, under a different formulation than given above, and have been widely applied in the study of computable analysis. Later, these were independently rediscovered by Dorais, Dzhafarov, Hirst, Mileti, and Shafer~\cite{DDHMS-2016}, and shown to be the uniform versions of computable reducibility and strong computable reducibility, respectively (see~\cite[Appendix A]{DDHMS-2016}).

The investigation of these notions has seen a recent surge of interest. (An updated bibliography is maintained by Brattka~\cite{Brattka-bib}.) Collectively, they provide a way of refining the analyses of effective and reverse mathematics, by elucidating subtler points of similarity and difference between various principles. In the case of Ramsey's theorem, one starting point of interest was in the number of colors. Over $\RCA_0$, the principle $\RT^n_k$ is equivalent to $\RT^n_\ell$ for all $k > \ell$, but the usual proof that $\RT^n_k$ is implied by $\RT^n_\ell$ uses $\RT^n_\ell$ multiple times, and as such does not fit any of the notions in Definition~\ref{def:reds}. Dorais et al.~\cite[Theorem 3.1]{DDHMS-2016} showed that if $k > \ell$ then $\RT^n_k \nsured \RT^n_\ell$, and this was subsequently improved by Hirschfeldt and Jockusch~\cite[Theorem 3.3]{HJ-TA} and Rakotoniaina \cite{Rakotoniaina-2015} (see \cite[Theorem 4.21]{BR-TA}) to show that also $\RT^n_k \nured \RT^n_\ell$. Thus, the proof of $\RT^n_k$ from $\RT^n_\ell$ is essentially non-uniform. Surprisingly, Patey~\cite[Corollary 3.15]{Patey-TA} showed that even $\RT^n_k \ncred \RT^n_\ell$. Thus, under any of the above reducibilities, what was basically a single principle in the classical framework is separated into infinitely many.

\subsection{Ramsey's theorem for singletons}

Our interest in this paper is in the principle $\RT^1_k$, and specifically, how it relates to the stable Ramsey's theorem and the cohesive principle. We begin with the former. A coloring $c : [\omega]^2 \to k$ is \emph{stable} if for every $x \in \omega$ there is an $s > x$ and an $i < k$ such that for all $y \geq s$, $c(x,y) = i$. In other words, the color of $c(x,y)$ is constant for all sufficiently large $y$. In this case, we write $\lim_y c(x,y) = i$.

\begin{highlight}{Stable Ramsey's theorem for $k$-colorings ($\SRT^2_k$)}
	For every stable coloring $c : [\omega]^2 \to k$, there is an infinite set $H$ which is homogeneous for $c$.	
\end{highlight}

\noindent It is convenient to define a set $X$ to be \emph{limit homogeneous} for a stable coloring $c : [\omega]^2 \to k$ if for some $i < k$, we have $\lim_y c(x,y) = i$ for all $x \in X$. In this case, we say $X$ is limit homogeneous \emph{with color $i$}. Every infinite homogeneous set is limit homogeneous (with the same color), but not conversely. Note that if $F$ is finite and homogeneous for $c$ with color $i$, then $F$ is extendible to an infinite homogeneous set if and only if it is also limit homogeneous with color $i$ and there are infinitely many $x$ such that $\lim_y d(x,y) = i$.

One can think of an instance $c : [\omega]^2 \to k$ of $\SRT^2_k$ as an instance $d : \omega \to k$ of $\RT^1_k$ defined by $d(x) = \lim_y c(x,y)$. And from any solution $X$ to $d$, one can thin out to obtain a solution $H$ to $c$, and $H \leq_T c \oplus X$. Of course, $d$ is not computable from $c$, merely from the Turing jump of $c$. Thus, while $\SRT^2_k$ is not computably reducible to $\RT^1_k$, it is computably equivalent to a kind of $\Delta^0_2$ version of $\RT^1_k$ called $\D^2_k$, which we will not discuss here. (See~\cite{CLY-2010} and~\cite[Section 3]{Dzhafarov-TA} for thorough explorations of how these principles are related.)

The second principle we will look at is the cohesive principle. A set $Y$ is \emph{cohesive} for a sequence $\seq{X_n : n \in \omega}$ of subsets of $\omega$ if for each $n$, either $Y \cap X_n$ or $Y \cap \overline{X_n}$ is finite.

\begin{highlight}{Cohesive principle ($\COH$)}
	For every sequence $\seq{X_n : n \in \omega}$ of subsets of $\omega$, there is an infinite set $Y$ which is cohesive for this sequence.
\end{highlight}

\noindent $\COH$, too, may be thought of in terms of $\RT^1_k$, namely, as a sequential form of $\RT^1_k$ with finite errors. To make this precise, define a set $Y$ to be \emph{almost homogeneous} for a coloring $c : \omega \to k$ if $Y - F$ is homogeneous for $c$ for some finite set $F$.

\begin{lemma}\label{lem:cohequiv}
	The following statements are equivalent under $\sured$.
	\begin{enumerate}
		\item $\COH$.
		\item For every sequence $\seq{c_k : k \geq 1}$ of colorings $c_k : \omega \to k+1$, there is an infinite set $Y$ which is almost homogeneous for each $c_k$.
	\end{enumerate}
\end{lemma}

\begin{proof}
	(Statement 2 $\sured$~$\COH$.) Fix a sequence $\seq{c_k : k \geq 1}$ of colorings $c_k : \omega \to k+1$. We define a sequence of sets $\seq{X_n : n \in \omega}$ as follows. Partition $\omega$ into adjacent finite intervals $B_1, B_2, \ldots$, with $|B_k| = \lceil \log_2 (k+1) \rceil$. Fix $k$, and suppose $B_k = \{n_0 < \ldots < n_{\lceil \log_2 (k+1) \rceil - 1}\}$; we define $X_{n_0}(x),\ldots,X_{n_{\lceil \log_2 (k+1) \rceil-1}}(x)$ for each $x \in \omega$. Since $c_k(x) \leq k$, it consists of at most $\lceil \log_2 (k+1) \rceil$ digits when written in binary, and so by prepending $0$s if necessary, we can encode $c_k(x)$ as a binary sequence $\sigma_{k,x}$ of length $\lceil \log_2 (k+1) \rceil$. For instance, if $c_5(2) = 3$ then $\sigma_{5,2} = \seq{011}$, and if $c_{10}(3) = 2$ then $\sigma_{10,3} = \seq{0010}$. We define $X_{n_j}(x) = \sigma_{k,x}(j)$ for each $j < \lceil \log_2 (k+1) \rceil$. Now if $Y$ is an infinite cohesive set for $\seq{X_n : n \in \omega}$, then for all sufficiently large $x$ the finite binary sequence $\seq{X_{n_0}(x),\ldots,X_{n_{\lceil \log_2 (k+1) \rceil-1}}(x)}$ is the same, and hence also $c_k(x)$ is the same. Thus, $Y$ is almost homogeneous for $c_k$.
	
	($\COH$ $\sured$~Statement 2.)  This is clear, by identifying sets with their  characteristic functions. \qedhere
\end{proof}

\noindent Since $\sured$ is the strongest of the reducibilities we are discussing, it follows that for our purposes, $\COH$ can be used interchangeably with Statement 2.

\subsection{Main theorems}

Clearly, for all $k$ we have that $\RT^1_k \sured \SRT^2_k$. Let $\RT^1_{<\infty}$ be the problem whose instances are colorings $c : \omega \to k$, for all $k \geq 1$, and solutions are, as before, infinite homogeneous sets. As a statement of second-order arithmetic, this corresponds to the statement $\forall k \geq 1~\RT^1_k$. Hirst~\cite[Theorem 6.8]{Hirst-1987} proved that $\RT^2_2 \to \RT^1_{<\infty}$ over $\RCA_0$, and his proof actually shows that $\RT^1_{<\infty} \sured \RT^2_2$. A modification of this proof shows also that $\SRT^2_2 \to \RT^1_{<\infty}$ over $\RCA_0$, but in terms of computability-theoretic reducibilities, it yields only that $\RT^1_{<\infty} \cred \SRT^2_2$. In fact, Hirschfeldt and Jockusch~\cite[Theorem 2.10 (4)]{HJ-TA} showed that for all $k$, $\RT^1_{k+1} \nured \SRT^2_k$. This leaves a gap around strong computable reducibility. More generally, it was asked in~\cite[Question 5.4]{HJ-TA} whether there exist $k > \ell$ such that $\RT^1_k \scred \SRT^2_\ell$. Our main result in this paper is a negative answer to this question. In fact, we will prove the following considerably stronger fact.

\begin{theorem}\label{thm:kl}
	If $k > \ell$, there exists a coloring $c : \omega \to k$ such that for every stable coloring $d : [\omega]^2 \to \ell$ (computable from $c$ or not), there is an infinite homogeneous set $H$ for $d$ such that $H$ computes no infinite homogeneous set for $c$.
\end{theorem}

\begin{corollary}
	If $k > \ell$, then $\RT^1_k \nscred \SRT^2_\ell$.	
\end{corollary}

\noindent Theorem \ref{thm:kl} can be viewed as saying that if $k > \ell$, then not only does $\RT^1_k$ not follow from $\SRT^2_\ell$ by any natural argument, of the kind encapsulated by the reductions of Definition \ref{def:reds}, but also there is a true combinatorial, rather than merely computability-theoretic, difference between the two. A similar ``combinatorial non-reduction'' was exhibited by Hirschfeldt and Jockusch \cite[Theorem 3.9]{HJ-TA} and Patey \cite[Corollary 3.4]{Patey-TA}, who constructed a coloring $c : \omega \to k$ such that for every stable coloring $d : [\omega]^2 \to \ell$, there is an infinite \emph{limit} homogeneous set $L$ for $d$ such that $L$ computes no infinite homogeneous set for $c$. Theorem \ref{thm:kl} is an extension of this fact, though the proof is not: the move from limit homogeneous sets to fully homogeneous ones in our case requires an entirely different set of techniques.

Theorem \ref{thm:kl} has an application to the study of the relative strength of the stable Ramsey's theorem and $\COH$. An important connection between these principles, due to Cholak, Jockusch, and Slaman~\cite[Section 3]{CJS-2001}, is that $\RT^2_k$ is equivalent to $\SRT^2_k + \COH$ over $\RCA_0$ (see~\cite[Corollary A.1.4]{Mileti-2004}). Only recently has the question of whether $\SRT^2_k$ implies $\COH$ been answered (in the negative), by Chong, Slaman, and Yang~\cite{CSY-TA}, but it remains open whether $\COH \leq_\omega \SRT^2_k$, and even whether $\COH \cred \SRT^2_k$. As a partial step towards a negative answer, Dzhafarov~\cite[Corollary 5.3]{Dzhafarov-TA} proved that $\COH \nscred \SRT^2_2$. In turn, it was asked in~\cite[Question 6.3]{Dzhafarov-TA} whether the same holds for $\SRT^2_k$ for $k > 2$. We give an affirmative answer to this question, again in a stronger form.

\begin{corollary}\label{cor:generalcoh}
	There is a family $\seq{c_k : k \geq 1}$ of colorings $c_k : \omega \to k+1$ such that for every stable coloring $d : [\omega]^2 \to \ell$ (computable from this family or not), there is an infinite homogeneous set $H$ for $d$ such that for some $k \geq 1$, $H$ computes no almost homogeneous set for $c_k$.
\end{corollary}

\begin{proof}
	By Theorem \ref{thm:kl}, for each $k \geq 1$ there is a coloring $c_k : \omega \to k+1$ such that for every stable $d : [\omega]^2 \to k$, there is an infinite homogeneous set $H$ for $d$ such that $H$ computes no infinite homogeneous set for $c_k$. Then, in particular, $H$ also computes no almost homogeneous set for $c_k$. Thus, the family $\seq{c_k : k \geq 1}$ is as desired.
\end{proof}

\begin{corollary}\label{cor:maincoh}
	$\COH \nscred \SRT^2_{<\infty}$.
\end{corollary}

\begin{proof}
	This is immediate by Lemma \ref{lem:cohequiv}.
\end{proof}

We do not know how effective we can choose the instance of $\COH$ witnessing Corollary \ref{cor:maincoh} to be, and in particular, whether or not we can find a computable such instance. In the terminology introduced by Jockusch and Stephan \cite[Section 1]{JS-1993}, the latter is equivalent to whether there is a computable stable coloring of pairs $d$, every homogeneous set for which has $p$-cohesive degree, which is in turn equivalent to the aforementioned open question of whether $\COH \cred \SRT^2_2$. More generally, we do not know if there is any set $X$, and any stable coloring of pairs $d$ computable from $X$, such that for every infinite homogeneous set $H$ for $d$, $X \oplus H$ has $p$-cohesive degree relative to $X$. Corollary \ref{cor:maincoh} shows that the answer is no if we ask for $H$ itself, rather than $X \oplus H$, to have $p$-cohesive degree relative to $X$.

By contrast, the instance of $\COH$ witnessing Corollary \ref{cor:generalcoh} cannot be chosen to be computable, nor even $\Delta^1_1$. Indeed, for every $\Delta^1_1$ set $X$, it is easy to define a stable coloring $d : [\omega]^2 \to 2$ (not necessarily computable from $X$) such that the principal function of any infinite homogeneous set $H$ for $d$ computes $X$. (By results of Solovay \cite{Solovay-1978} and Groszek and Slaman \cite{GS-TA}, a set $X$ is $\Delta^1_1$ if and only if it has a modulus of computation, meaning a function $f : \omega \to \omega$ such that $X$ is computable from any function $g : \omega \to \omega$ that dominates $f$. Given a modulus $f$ for $X$, we set $d(x,y) = 0$ if $y-x \leq \max \{ f(z) : z \leq x+1\}$, and set $d(x,y) = 1$ otherwise. Then $\lim_y d(x,y) = 1$ for all $x$, so every infinite homogeneous set $H$ for $d$ has color $1$. But for any such $H$, $H(x+1) > f(x+1)$ for all $x$.) In particular, for every $\Delta^1_1$ set $X$, there is a stable coloring of pairs $d$, every infinite homogeneous set for which can compute $X'$, and hence can compute a solution to any $X$-computable instance of $\COH$. Hence, for every $\Delta^1_1$ set $X$ there is a stable coloring of pairs $d$, every infinite homogeneous set for which has $p$-cohesive degree relative to $X$.

The outline of the rest of the paper is as follows. In Section \ref{sec:setup}, we introduce the forcing notions we need for the proof of Theorem \ref{thm:kl}, and we prove the theorem modulo a key diagonalization step, Lemma \ref{lem:kl}. This lemma relies on a significant simplification and extension of the \emph{tree labeling method} of constructing homogeneous sets, introduced in~\cite[Section 5]{Dzhafarov-TA} to prove that $\COH \nscred \SRT^2_2$. We review the tree labeling method in Section \ref{sec:proof}, and then conclude by proving the lemma. Our notation in the sequel is standard, with the following exception. For a Turing functional $\Delta$, we write $\Delta^X(x) \simeq y$ to denote that either $\Delta^X(x)$ diverges or $\Delta^X(x) \downarrow = y$. For a finite set $F$, we follow the convention that if $\Delta^F(x) \downarrow$ then the computation halts with use bounded by $\max F$.

\section{Forcing notions and outline of proof}\label{sec:setup}

In this section, we prove Theorem \ref{thm:kl} modulo a key combinatorial lemma, Lemma \ref{lem:kl}, which we delay until the next section. Given $k > \ell$, we need to build an instance $c : \omega \to k$ of $\RT^1_k$, and for every stable coloring $d : [\omega]^2 \to \ell$, an infinite homogeneous set $H$ such that $H$ computes no infinite homogeneous set for $c$. We obtain each of $c$ and $H$ as a generic for a suitable notion of forcing: Cohen forcing in the case of $c$, and Mathias forcing in the case of $H$. We begin by defining the relevant forcing notions.

Throughout, let $M$ be a fixed countable transitive model of $\ZFC$. Let $\mathbb{C}_k$ be Cohen forcing with strings $\sigma \in k^{<\omega}$, so that a generic is a coloring $\omega \to k$. We let the desired instance $c$ of $\RT^1_k$ be a generic for $\mathbb{C}_k$ over $M$.

To define the homogeneous set $H$ for the given stable coloring $d : [\omega]^2 \to \ell$, we first recall the definition of Mathias forcing, which is frequently employed in the study of Ramsey's theorem. (See \cite{CDHS-2014,CDS-TA} for a general discussion of Mathias forcing in computability theory.) Here, conditions are pairs $(E,I)$ such that $E$ is a finite set, $I$ is an infinite set called a \emph{reservoir}, and $E < I$. A condition $(E',I')$ \emph{extends} $(E,I)$, denoted $(E',I') \leq (E,I)$, if $E \subseteq E' \subseteq E \cup I$ and $I' \subseteq I$. Frequently, the reservoirs $I$ in a particular Mathias forcing construction are restricted to a certain family of sets, for example the (infinite) computable sets, as in Cholak, Jockusch, and Slaman~\cite[Section 4]{CJS-2001}. In our case, we let $\mathcal{I}$ be the set of all infinite subsets of $\omega$ in the model $M$, and restrict to working with conditions $(E,I)$ with $I \in \mathcal{I}$.

For every~$i < k$ and~$j < \ell$, let~$\mathcal{I}_{i,j}$ be the collection of sets~$I \in \mathcal{I}$
such that for all increasing map $f \in M$ with $\ran(f) \in \mathcal{I}$, if $\ran(f) \subseteq I$, 
then there is some~$w \in \dom(f)$ such that~$c(w) = i$ and~$\lim_y d(f(w), y) = j$.
Note that~$\mathcal{I}_{i,j}$ is upward-closed in $(\mathcal{I}, \supseteq)$.

\begin{lemma}\label{lem:avoid-color}
For every~$i < k$, the set~$\mathcal{I}_i = \bigcup_{j < \ell} \mathcal{I}_{i,j}$ is dense in $(\mathcal{I}, \supseteq)$.
\end{lemma}
\begin{proof}
Suppose that $\mathcal{I}_i$ is not dense and let~$I \in \mathcal{I}$ have no extension in~$\mathcal{I}_i$.
Define a finite sequence of increasing maps $f_0, \dots, f_{\ell-1} \in M$
and a finite sequence of sets $I = I_0 \supseteq \dots \supseteq I_{\ell} \in \mathcal{I}$
such that for each~$j < k$, $\ran(f_j) = I_{j+1}$
and $f_j$ witness that~$I_j$ has no extension in~$\mathcal{I}_{i,j}$,
that is, for every~$w \in \dom(f)$, if~$c(w) = i$ then~$\lim_y d(f(w), y) \neq j$.
For each~$x \in I_\ell$, let~$F_x = \{ f^{-1}_j(x) : j < \ell \}$.
The set~$P = \{ F_x : x \in I_{\ell}$ belongs to $M$,
so by $M$-genericity of~$c$, $c(F_x) = \{i\}$ for some~$x \in I_{\ell}$.
By choice of~$F_x$, $\lim_y d(x, y) \not \in \ell$. Contradiction.
\end{proof}

Let~$\LimSet \in \bigcap_{i < k} \mathcal{I}_i$ and let~$j_i < \ell$ be such that~$I_d \in \mathcal{I}_{i,j}$ for each~$i < k$.
Since~$k > \ell$, there are some~$i_0 < i_1 < k$ such that $j_{i_0} = j_{i_1}$. Let~$j = j_{i_0} = j_{i_1}$.
From now on, $\LimSet, i_0, i_1$ and~$j$ are fixed. The following lemma is a useful consequence
of the choices of~$\LimSet$ and~$j$.

\begin{lemma}\label{lem:infcolors}
For all $I \in \mathcal{I}$, if $I \subseteq \LimSet$ then $\lim_y d(x,y) = j$ for infinitely many $x \in I$.
\end{lemma}
\begin{proof}
Fix some~$I \in \mathcal{I}$ such that $I \subseteq \LimSet$
and fix some~$n \in \omega$. Define~$f : \omega \to I \setminus [0,n]$
be the increasing map which to $x$ associates the $x$th element of $I \setminus [0,n]$.
Since $\LimSet \in \mathcal{I}_{i_0, j}$ and~$f \in M$, there is some~$w \in \dom(f) = \omega$
such that~$c(w) = i_0$ and~$\lim_y d(f(w), y) = j$. In particular~$f(w) > n$,
so there are an unbounded number of~$x \in I$ such that~$\lim_y d(x, y) = j$.
\end{proof}

We are now ready to define our notion of forcing.

\begin{definition}
We define $\mathbb{M}_{d,\LimSet,j}$ to be the following notion of forcing. A \emph{condition} is a Mathias condition $(E, I)$ such that $I \in \mathcal{I}$ with $I \subseteq \LimSet$, and $E$ is homogeneous for $d$ with color $j$, and $d(x,y) = j$ for all $x \in E$ and all $y \in I$. A condition $(E', I')$ \emph{extends} $(E, I)$, denoted $(E', I') \leq (E, I)$, if $(E',I') \leq (E,I)$ as a Mathias condition.
\end{definition}

In particular, if $(E, I)$ is an $\mathbb{M}_{d,\LimSet,j}$ condition then $E$ is limit homogeneous for $d$ with color $j$, and $E \cup \{y\}$ is homogeneous with color $j$ for each $y \in I$. A generic filter for this forcing thus yields a subset~$H$ of $\LimSet$
 which is homogeneous for $d$ with color $j$.

\begin{lemma}\label{lem:Hjinf}
If $H$ is generic for $\mathbb{M}_{d,\LimSet,j}$, then $H$ is infinite.
\end{lemma}
\begin{proof}
	Fix any condition $(E, I)$. By Lemma~\ref{lem:infcolors}, there is some~$x \in I$ such that $\lim_y d(x,y) = j$.
Let $m > x$ be such that $d(x,y) = j$ for all $y \geq m$. Let $E' = E \cup \{x\}$, and let $I' = \{x \in I: x > m\}\}$, which is a co-initial segment of $I$ and hence belongs to $\mathcal{I}$. Then $(E', I') \leq (E, I)$ and $|E'| = |E| + 1$. Thus, it is dense to add an element to $H$, so by genericity, $H$ is infinite.
\end{proof}

We fix a countable transitive model $M'$ of $\ZFC$ with $M \cup \{c, d\} \subseteq M'$, and choose the set $H$ to be generic for $\mathbb{M}_{d,\LimSet,j}$ over $M'$. 
The following lemma, whose proof we give in the next section, will allow us to complete the proof of Theorem \ref{thm:kl}.

\begin{lemma}\label{lem:kl}
	Let $\Delta$ be a Turing functionals Then for each $i \in \{i_0, i_1\}$, one of the following holds:
	\begin{enumerate}
		\item $\Delta^H$ is not (the characteristic function of) an infinite set;
		\item there is a $w \in \omega$ such that $\Delta^H(w) \downarrow = 1$ and $c(w) = i$.
	\end{enumerate}
\end{lemma}

\noindent The theorem is now an immediate consequence.

\begin{proof}[Proof of Theorem \ref{thm:kl}]
	We claim that $H$ computes no infinite homogeneous set for $c$. By Lemma \ref{lem:Hjinf}, $H$ is an infinite homogeneous set for $d$, so this suffices. Seeking a contradiction, suppose not. Let $\Delta$ be a Turing functional such that $\Delta^H$ is an infinite homogeneous set for $c$. By Lemma \ref{lem:kl}, for each $i \in \{i_0, i_1\}$ we can find a number $w$ in the set $\Delta^H$ such that $c(w) = i$,
so $\Delta^H$ is not homogeneous for $c$ after all.
\end{proof}

\section{Proof of Lemma \ref*{lem:kl}}\label{sec:proof}

As mentioned in the introduction, the proof of Lemma \ref{lem:kl} employs the so-called tree labeling method, introduced in \cite{Dzhafarov-TA}. As this method is new, we begin this section with a careful presentation of this method.

\subsection{Tree labeling}

Let $\lambda$ denote the empty string. For a non-empty string $\alpha$, we let $\trunc{\alpha} = \alpha \upharpoonright |\alpha|-1$. That is, $\trunc{\alpha}$ is the string formed by removing the last element of $\alpha$. Given a string $\beta$, we write $\alpha * \beta$ for the concatenation of $\alpha$ by $\beta$, and given $x \in \omega$, we write $\alpha*x$ for the concatenation of $\alpha$ by the singleton sequence $\seq{x}$. We call $\alpha*x$ a \emph{successor} of $\alpha$. Note that $\trunc{(\alpha*x)} = \alpha$.

\begin{definition}
	Fix $n \in \omega$, and let $\Delta$ be a Turing functional and $(E,I)$ a Mathias condition. We define $T(n,\Delta,E,I) \subseteq I^{< \omega}$ by $\lambda \in T(n,\Delta,E,I)$ and for a non-empty string $\alpha$, $\alpha \in T(n,\Delta,E,I)$ if $\alpha \in I^{< \omega}$ is increasing and 
	\[
		\forall F \subseteq \ran(\trunc{\alpha}) \, \forall w \geq n \, ( \Delta^{E \cup F}(w) \simeq 0).
	\] 
\end{definition}

It is clear from the definition that $T(n,\Delta,E,I)$ is closed under initial segments, and so is a tree. In what follows, we will always interpret the functional $\Delta$ in these trees to have output values restricted to $\{ 0,1 \}$.

\begin{lemma}\label{lem:Tprops} 
$T = T(n,\Delta,E,I)$ has the following properties. 
\begin{enumerate}
\item If $T$ is not well-founded and $P$ is any infinite path through $T$, then $\ran(P)$ is infinite and $\Delta^{E \cup F}(w) \simeq 0$ for all $w \geq n$ and all $F \subseteq \ran(P)$.
\item If $\alpha \in T$ is not terminal, then $\forall x \in I \, (x > \ran(\alpha) \rightarrow \alpha*x \in T)$.
\item If $\alpha \in T$ is terminal, then there is an $F \subseteq \ran(\alpha)$ and a $w \geq n$ such that $\Delta^{E \cup F}(w) \downarrow = 1$. In particular, if $T$ consists of just the root node then $\Delta^E(w) \downarrow = 1$ for some $w \geq n$.
\end{enumerate}
\end{lemma}

\begin{proof}
Property (1) follows immediately from the definition of $T(n,\Delta,E,I)$. For property (2), let $\alpha \in T$ be a non-terminal node and let $\alpha*x$ be a successor of $\alpha$ in $T$. By definition, for every $F \subseteq \ran(\trunc{(\alpha*x)}) = \ran(\alpha)$ and every $w \geq n$, we have $\Delta^{E \cup F}(w) \simeq 0$. But this fact is independent of $x$. Hence, so long as $x \in I$ and $x > \ran(\alpha)$, so that $\alpha*x$ is an increasing sequence from $I$, we have $\alpha*x \in T$. For property (3), let $\alpha$ be a terminal node. Because $I$ is infinite, there is an $x \in I$ with $x > \ran(\alpha)$. And since $\alpha*x \not \in T$, there is $F \subseteq \ran(\trunc{(\alpha*x)}) = \ran(\alpha)$ and a $w \geq n$ such that $\Delta^{E \cup F}(w) \downarrow = 1$.
\end{proof}

Our main concern in the proof of Lemma \ref{lem:kl} will be when $T$ is well-founded. If this is the case, we label the nodes of $T(n,\Delta,E,I)$ and prune to a more uniformly labeled subtree.

\begin{definition}
	Let $T = T(n,\Delta,E,I)$ be well-founded. We label the nodes of $T$ by recursion starting at the terminal nodes. The label of each node will be either a number $w$ or the symbol $\infty$. If $\alpha \in T$ is terminal, then for some $w \geq n$ there is an $F \subseteq \ran(\alpha)$ such that $\Delta^{E \cup F}(w) \downarrow = 1$, and we label $\alpha$ by the least such $w$. Now suppose $\alpha \in T$ is not terminal, and assume by recursion that every successor of $\alpha$ in $T$ has been labeled. If there is a number $w$ such that infinitely many of the successors of $\alpha$ are labeled $w$, we choose the least such $w$ and label $\alpha$ with $w$. Otherwise, we label $\alpha$ with $\infty$.
\end{definition}

Thus, each non-terminal $\alpha \in T$ either has a numerical label $w$, in which case so do infinitely many of its successors in $T$, or $\alpha$ has label $\infty$. In the latter case, each number $w$ that appears as the label of any successor of $\alpha$ can label at most finitely many other successors of $\alpha$.

\begin{definition}
	Let $T = T(n,\Delta,E,I)$ be well-founded. We define the \emph{labeled subtree} $T^L = T^L(n,\Delta,E,I)$ of $T$ as follows, starting at the root of $T$ and working down by levels. In all cases, if a node of $T$ is placed in $T^L$ then it retains its label from $T$. To begin, the root of $T$ is placed in $T^L$. Now suppose $\alpha \in T$ has been placed in $T^L$. If $\alpha$ has a numerical label $w$, then each successor of $\alpha$ in $T$ with label $w$ is placed in $T^L$. If $\alpha$ has label $\infty$ and infinitely many of its successors have label $\infty$, then each successor of $\alpha$ in $T$ with label $\infty$ is placed in $T^L$. If $\alpha$ has label $\infty$ and only finitely many of its successors have label $\infty$, then $\alpha$ must have successors labeled with infinitely many different numerical values. For each numerical value $w$ which labels some successor of $\alpha$, we pick the least $x$ such that $\alpha*x$ has label $w$ and place $\alpha*x$ in $T^L$.
\end{definition}

The next lemma lists properties of $T^L$ which are straightforward to verify.

\begin{lemma}
\label{lem:pruned}
Assume $T = T(n,\Delta,E,I)$ is well-founded. A node $\alpha \in T^L$ is terminal in $T^L$ if and only if it is terminal in $T$. 
Each non-terminal node $\alpha \in T^L$ has infinitely many successors in $T^L$, and these successors are either all labeled with the same numerical value $w$ (if $\alpha$ has label $w$), are all labeled $\infty$ (and $\alpha$ is labeled $\infty$), or each successor has a distinct numerical label (and $\alpha$ is labeled $\infty$). 
\end{lemma}

Suppose $T = T(n,\Delta,E,I)$ is well-founded. For each non-terminal $\alpha \in T^L$, we define the \emph{row below $\alpha$} to be the infinite set $\{x : \alpha * x \in T^L\}$. In the event that each $\alpha*x$ has a distinct numerical label, we also define the \emph{labeled row below $\alpha$} to be $\{ \seq{x,w} : \alpha*x \in T^L \wedge \alpha*x \text{ has label } w\}$. Thus, the labeled row below $\alpha$ is defined if and only if $\alpha$ has label $\infty$ but each of its successors has a numerical label. We call a set a \emph{row of $T^L$} or \emph{labeled row of $T^L$} if it is the row or labeled row below some non-terminal $\alpha \in T^L$. If $Y$ is a labeled row of $T^L$, then $\{x : \exists w~\seq{x,w} \in Y\}$ is a row of $T^L$, and we denote it by $\pi_1 Y$. The following observation will serve as a crucial connection with the forcing notions defined in the previous section.

\begin{observation}\label{obs:main}
	Recall the collection $\mathcal{I}$ of infinite subsets of $\omega$ in our fixed countable transitive model $M$ of $\ZFC$, as well as the model $M'$ extending $M$ and containing the generic $c$ and the given stable coloring $d : [\omega]^2 \to \ell$. If $I \in \mathcal{I}$ then $T = T(n,\Delta,E,I)$ belongs to $M$. If $T$ is not well-founded, it follows by $\Pi^1_1$ absoluteness that $T$ is not well-founded in $M$, and hence $M$ contains an infinite path through $T$. The range of this path thus belongs to $\mathcal{I}$. And if $T$ is well-founded, then every row and every labeled row of $T^L$ belongs to $M$, and in particular, every row of $T^L$ belongs to $\mathcal{I}$.
\end{observation}

\subsection{Proof of Lemma \ref*{lem:kl}}

Let $M$, $\mathcal{I}$, $c$, $d$, $M'$, $I_d$, $i_0$, $i_1$ and~$j$ be as in Section \ref{sec:setup}. We will need the technical result below, which gives the most important application of the tree labeling method for our purposes.

\begin{lemma}\label{lem:row}
	Fix $I \in \mathcal{I}$ with $I \subseteq \LimSet$, and assume $T = T(n,\Delta,E,I)$ is well-founded. For each $\beta \in T^L$, and each $m \in \omega$, there is a $\gamma \in \omega^{<\omega}$ with the following properties:
	\begin{itemize}
	\item $\gamma = \lambda$ or $m \leq \min \ran(\gamma)$;
	\item $\beta * \gamma$ is a terminal node of $T^L$;
	\item $\ran(\gamma)$ is homogeneous and limit homogeneous for $d$ with color $j$.
	\end{itemize}
\end{lemma}

\begin{proof}
	Fix $\beta \in T^L$. We define a sequence $\gamma_0 \preceq \gamma_1 \preceq \cdots$ such that for each $s$, $\beta * \gamma_s \in T^L$, $m \leq x$ for each $x \in \ran(\gamma_s)$, and $\ran(\gamma_s)$ is homogeneous and limit homogeneous for $d$ with color $j$. Furthermore, we ensure that if $\beta*\gamma_s$ is not terminal in $T^L$ then $\beta*\gamma_{s+1}$ is a successor of $\beta*\gamma_s$. Thus, since $T^L$ is well-founded, $\beta*\gamma_s$ must be terminal in $T^L$ for some $s$, and we can then take $\gamma = \gamma_s$. To construct the sequence, let $\gamma_0 = \lambda$, and suppose inductively that we have defined $\gamma_s$. If $\beta*\gamma_s$ is terminal in $T^L$, we are done. Otherwise, let $m'$ be large enough so that $d(x,y) = j$ for all $x \in \ran(\gamma_s)$ and all $y \geq m'$. By Observation \ref{obs:main}, the row below $\beta*\gamma_s$ belongs to $\mathcal{I}$, and it is an infinite subset of $I$ and hence of $\LimSet$. Thus, there must be infinitely many $x$ in this row such that $\lim_y d(x,y) = j$. In particular, we can choose some such $x$ with $x \geq \max\{m,m'\}$, and we let $\gamma_{s+1} = \gamma_s * x$.
\end{proof}

We can now prove Lemma \ref{lem:kl}, which completes the argument.

\begin{highlight}{Lemma \ref*{lem:kl}}
	Let $\Delta$ be a Turing functionals Then for each $i \in \{i_0, i_1\}$, one of the following holds:
	\begin{enumerate}
		\item $\Delta^H$ is not (the characteristic function of) an infinite set;
		\item there is a $w \in \omega$ such that $\Delta^H(w) \downarrow = 1$ and $c(w) = i$.
	\end{enumerate}
\end{highlight}

\begin{proof}
	Let $\Delta$ and $i \in \{i_0, i_1\}$ be given. Fix any $\mathbb{M}_{d,\LimSet,j}$ condition $(E, I)$. We seek an extension $(E', I')$ such that one of the following holds:
	\begin{enumerate}
		\item there is an $n \in \omega$ such that for all $F \subseteq I'$ and all $w \geq n$, $\Delta^{E' \cup F}(w) \simeq 0$;
		\item there is a $w \in \omega$ such that $\Delta^{E'}(w) \downarrow = 1$ and $c(w) = i$.
	\end{enumerate}

	Since $d \in M'$, also $\mathbb{M}_{d,\LimSet,j} \in M'$. And since also $c \in M'$ and $H$ is generic over $M'$, this suffices.
	
	\begin{highlight}{Construction}
	For each $n$, let $T_n = T(n, \Delta, E, I)$, and if $T_n$ is well-founded, let $T^L_n = T^L_n(n, \Delta,E,I)$. We consider the following cases.
	\end{highlight}
	
	\begin{highlight}{Case 1}
		There is an $n$ such that $T_n$ is not well-founded. 	
	\end{highlight}

	Let $P$ be any infinite path through $T_n$ in $M$, and let $I' = \ran(P)$. By Observation \ref{obs:main}, $I' \in \mathcal{I}$. We define $E' = E$, so that $(E', I')$ extends $(E, I)$. Clearly, this extension satisfies clause (1) above.
	
	\begin{highlight}{Case 2}
		For all $n$, $T_n$ is well-founded, and the root of $T^L_n$ has a numerical label~$w_n$. 
	\end{highlight}

	By Observation \ref{obs:main}, the set of all $w_n$ belongs to $M$, so by genericity of $c$, there is an $n$ such that $c(w_n) = i$. We apply Lemma \ref{lem:row} with $\beta = \lambda$ to obtain a terminal node $\gamma$ of $T^L_n$ such that $\ran(\gamma)$ is homogeneous and limit homogeneous for $d$ with color $j$. Then $\gamma$ is labeled by $w_n$ in $T^L_n$, so there is a finite $F \subseteq \ran(\gamma)$ such that $\Delta^{E \cup F}(w_n) \downarrow = 1$, and we define $E' = E \cup F$. (Since $F \subseteq I$, we have $E < F$. By Lemma \ref{lem:Tprops}, if $T^L_n$ consists of just the root node, then $\beta = \lambda$ and $F = \emptyset$.) Note that $E'$ is homogeneous and limit homogeneous for $d$ with color $j$. We choose $m > F$ such that $d(x,y) = j$ for all $x \in F$ and all $y \geq m$, and define $I' = \{x \in I: x > m\}$, so that $I' \in \mathcal{I}$. Now $(E', I')$ is an extension of $(E, I)$ satisfying clause (2).

	\begin{highlight}{Case 3}
		For some $n$, $T_n$ is well-founded, and the root of $T^L_n$ has label $\infty$. 
	\end{highlight}
	
	Let $\gamma$ be a terminal node of $T^L_n$ obtained by applying Lemma \ref{lem:row} with $\beta = \lambda$. Let $\alpha$ be an initial segment of $\gamma$ such that $\alpha$ has label $\infty$ and each successor of $\alpha$ in $T^L_n$ has numerical label. In particular, $\ran(\alpha)$ is homogeneous and limit homogeneous for $d$ with color $j$. By Observation \ref{obs:main}, the labeled row below $\alpha$ is in $M$, so there are infinitely many pairs $\seq{x,w}$ in this labeled row with $c(w) = i$. Let $m$ be such that $d(x,y) = j$ for all $x \in \ran(\alpha)$ and all $y \geq m$. 
	
	Let $f \in M$ be the increasing map defined by~$f(w) = x$ iff $\seq{x,w} \in \alpha$ and~$x \geq m$. In particular, $\ran(f) \in \mathcal{I}$ and $\ran(f) \subseteq I \subseteq I_d$. Since $I_d \in \mathcal{I}_{i,j}$, there is some~$w \in \dom(f)$ such that~$c(w) = i$
and~$\lim_y d(f(w), y) = j$. Let~$x = f(w)$. We have $\seq{x, w} \in \alpha$, $x \geq m$ and~$\lim_y d(x,y) = j$. We then apply Lemma \ref{lem:row} again, this time with $\beta = \alpha*x$, and find a $\delta$ such that $\delta = \lambda$ or $m \leq \min \ran(\delta)$, $\alpha*x*\delta$ is terminal in $T^L_n$, and $\ran(\delta)$ is homogeneous and limit homogeneous for $d$ with color $j$. Now $\alpha * x * \delta$ is labeled $w$ in $T^L_n$, and its range is homogeneous and limit homogeneous for $d$ with color $j$. We choose a finite $F \subseteq \ran(\alpha * x * \delta)$ such that $\Delta^{E \cup F}(w) \downarrow = 1$, and define $E' = E \cup F$. Let $m' > F$ be such that $d(z,y) = j$ for all $z \in \ran(F)$ and all $y \geq m'$, and define $I' = \{x \in I: x > m'\}$. Then $(E', I')$ extends $(E, I)$ and satisfies clause (2).
This completes the proof. \qedhere
\end{proof}

\bibliographystyle{plain}
\bibliography{Papers.bib}

\end{document}